\documentclass[12pt]{amsart}
\usepackage{amsmath,amsthm,amscd,amssymb}
\usepackage[margin=1.25in]{geometry}
\usepackage{color}

\theoremstyle{plain}

\theoremstyle{definition}

\theoremstyle{remark}

\theoremstyle{plain}
\newtheorem{theorem}{Theorem}[section]
\newtheorem{lemma}[theorem]{Lemma}

\newtheorem{corollary}[theorem]{Corollary}

\theoremstyle{definition}

\newtheorem{remark}[theorem]{Remark}

\newcommand{\bR}{\mathbb{R}}
\newcommand{\sq}{\square}

\begin{document}

\title{A Differential Harnack Inequality for the Newell-Whitehead Equation} 

\author[Booth]{Derek Booth}
\address{Department of Mathematics, Harvard University \\ Cambridge, MA 02138, USA}
\email{derekbooth@college.harvard.edu}

\author[Burkart]{Jack Burkart}
\address{Department of Mathematics, Stony Brook University \\ Stony Brook, NY 11794, USA}
\email{jack.burkart@stonybrook.edu}
 
\author[Cao]{Xiaodong Cao}
\address{Department of Mathematics, Cornell University \\ Ithaca, NY 14853-4201, USA}
\email{cao@math.cornell.edu}

\author[Hallgren]{Max Hallgren}
\address{Department of Mathematics, Cornell University \\ Ithaca, NY 14853-4201, USA}
\email{meh249@cornell.edu}

\author[Munro]{Zachary Munro}
\address{Department of Mathematics, Cornell University \\ Ithaca, NY 14853-4201, USA}
\email{zym2@cornell.edu}

\author[Snyder]{Jason Snyder}
\address{Department of Mathematics, Cornell University \\ Ithaca, NY 14853-4201, USA}
\email{jss474@cornell.edu}

\author[Stone]{Tom Stone}
\address{Department of Mathematics, University of Wisconsin-Madison \\ Madison, WI 53706, USA}
\email{tdstone@wisc.edu}

\date{\today}
\begin{abstract}
This paper will develop a Li-Yau-Hamilton type differential Harnack estimate for positive solutions to the Newell-Whitehead equation on $\bR^n$. We then use our LYH-differential Harnack inequality to prove several properties about positive solutions to the equation, including deriving a classical Harnack inequality, and characterizing standing solutions and traveling wave solutions.
\end{abstract}
\maketitle
\section{Introduction}
Consider any positive solution $f: \bR ^n \times [0, \infty) \rightarrow \bR$ to the Newell-Whitehead Equation, 
\begin{equation}\label{eq:1}
f_t = \Delta f + af - bf^3,
\end{equation}

here, we assume $a > 0, b > 0$. This equation was first introduced by A. C. Newell and J. A. Whitehead in 1969 \cite{nw69}, and was later studied by L. Segel \cite{segel}. Exact solutions to the equation were computed using the Homotopy Perturbation method by S. Nourazar, M. Soori, and A. Nazari-Golshan in 2011 \cite{nsg}, while some approximate solutions were computed in 2015 by J. Patade and S. Bhalekar \cite{pb2015}. The equation is an example of a reaction-diffusion equation, as it is used to model the change of concentration of a substance, given any chemical reactions that the substance may be undergoing (modeled by the $af - bf^3$ term), and any diffusion causing the chemical to spread throughout the medium (modeled by the $\Delta f$ term). More specifically, the Newell-Whitehead equation models Rayleigh-B$\acute{\text{e}}$nard convection, a reaction-diffusion phenomenon that occurs when a fluid is heated from below. \\
\indent In this paper, we are just concerned with positive solutions on $\bR^n$. For further discussion about working with functions on closed manifolds or complete non-compact manifolds, see \cite{CLPW14}. Our main theorem, Theorem 1.1, will outline a Li-Yau-Hamilton type differential Harnack estimate (2) that we will prove based on computing time-evolutions of the relevant quantities, see Hamilton \cite{hamilton11}. In the following, \textit{Harnack inequality} or \textit{Harnack estimate} refers to an LYH-type differential Harnack inequality. As an application, we will integrate our estimate (2) along a space time curve to obtain a classical Harnack inequality (16), see Corollary 4.1. Then we will use our \textit{Harnack estimate} to characterize both traveling wave solutions and standing solutions to the Newell-Whitehead equation.

\begin{theorem}
With $f > 0$ a solution to \eqref{eq:1}, define $l = \log f$.  Then:
\begin{equation}\label{H}
H =\alpha \Delta l + \beta |\nabla l|^2 +\gamma e^{2l} + \varphi(t) \geq 0,
\end{equation}
provided the following three inequalities hold:\\
\\
\begin{tabular}{ l l }
  (a) & $\alpha > \beta \geq 0$, \\
  (b) & $\gamma \leq \dfrac{-nb\alpha^2(2\alpha+\beta)}{3n\alpha^2-2(\alpha-\beta)\beta} < 0$,\\
  (c) & $4\gamma(\alpha-\beta)+n\alpha^2b<0$, \\
\end{tabular}\\
with $\varphi(t) = \left(\dfrac{a\alpha}{1-e^{2at}}\right)\left(\dfrac{\gamma}{\alpha b}e^{2at}-\dfrac{\alpha \gamma n}{4\gamma(\alpha-\beta)+\alpha^2bn}\right)$.
\\
If, instead of inequality (c), we have:\\
\begin{tabular}{ l l }
  (d) & $4\gamma(\alpha-\beta)+n\alpha^2b\geq0,$ \\
\end{tabular}\\
then:
\begin{equation}
H =\alpha \Delta l + \beta |\nabla l|^2 +\gamma e^{2l} + \psi(t) \geq 0,
\end{equation}
for:
\[\psi(t) = \begin{cases}
\dfrac{n\alpha^2}{2(\alpha-\beta)t} & t \leq T := \dfrac{n\alpha^2}{2(\alpha-\beta)(-a\gamma)}\left(2\left(\dfrac{\alpha-\beta}{n\alpha^2}\right)\gamma+b\right), \\
\dfrac{-an\alpha^2\gamma\left(e^{2a(t-T)}+1\right)}{n\alpha^2b\left(e^{2a(t-T)}+1\right)+4\gamma(\alpha-\beta)} & t>T.
\end{cases}\]
\end{theorem}

\begin{remark}
Condition $(a)$ of Theorem $1.1$ says that we are allowed to choose $\beta=0$. While our proof of this theorem will require $\beta > 0$, we can take $\beta \to 0$ at the end. 
\end{remark}

\begin{remark}
The quantity $H$ defined in (2) and (3) is referred to as a (LYH-differential) Harnack quantity. 
\end{remark}
\begin{remark}
Inequalities (2) and (3) are called differential Harnack inequalities because they involve derivatives of $f$, and integration along space-time paths leads to a comparison of the function $f$ at different points in space and time.
\end{remark}
\section{Li-Yau-Hamilton type differential Harnack Inequality}

We begin by calculating the evolution of the Harnack quantity $H$. For notational convenience, we introduce the box operator $\sq g(x,t) := g_t - \Delta g$. Our first lemma is the following:
\begin{lemma}
With $H$ defined as in \eqref{H}, we have:
\begin{align}
\sq H &= 2\nabla l \cdot \nabla H + 2(\alpha-\beta) |\nabla \nabla l|^2 + \varphi_t - \Delta \varphi-2\nabla l \cdot \nabla \varphi\\
&- 2be^{2l}\bigg[(H - \varphi) + 2\alpha |\nabla l|^2 + \beta |\nabla l|^2-\gamma\frac{a}{b} + 3\frac{\gamma}{b}|\nabla l|^2\bigg]\nonumber.
\end{align}
\end{lemma}

\begin{proof}
We begin by calculating the evolution quantities of the components of $H$.
\begin{equation}
\sq (\Delta l) = \Delta |\nabla l|^2 - 2b \Delta l e^{2l} - 4b |\nabla l|^2 e^{2l}.
\end{equation}
\begin{equation}
\sq( |\nabla l|^2) = 2\nabla l \cdot \nabla (\Delta l) +2 \nabla l \cdot \nabla (|\nabla l|^2)-4b |\nabla l|^2 e^{2l} - \Delta (|\nabla l|^2).
\end{equation}
\begin{equation}
\sq(e^{2l}) = 2 \nabla l \cdot \nabla (e^{2l})+2ae^{2l} - 2be^{4l} - 6|\nabla l|^2e^{2l}.
\end{equation}

So, using these in the evolution equation for $H$:
\begin{eqnarray*}
\sq H &=& \alpha \sq (\Delta l) + \beta \sq(|\nabla l|^2) + \gamma \sq(e^{2l})+\sq \varphi\\
&=& \alpha \bigg(\Delta |\nabla l|^2 - 2b \Delta l e^{2l} - 4b |\nabla l|^2 e^{2l}\bigg)\\
&+& \beta \bigg(2\nabla l \cdot \nabla (\Delta l) +2 \nabla l \cdot \nabla (|\nabla l|^2)-4b |\nabla l|^2 e^{2l} - \Delta (|\nabla l|^2)\bigg)\\
&+& \gamma\bigg(2 \nabla l \cdot \nabla (e^{2l})+2ae^{2l} - 2be^{4l} - 6|\nabla l|^2e^{2l}\bigg) + \varphi_t - \Delta \varphi.
\end{eqnarray*}

Now, by using the Weitzenbock-Bochner formula for $\bR ^n$:
\begin{equation}
\Delta(|\nabla l|^2) = 2\nabla l \cdot \nabla(\Delta l)+2|\nabla \nabla l|^2.
\end{equation}
the expression can be simplified and the lemma follows.
\end{proof}

Using the Cauchy-Schwarz inequality, we can show that the Harnack quantity satisfies the following inequality.

\begin{lemma}
The following inequality holds:
\begin{align}
\sq H &\geq  2\nabla l \cdot \nabla H + H\bigg[2\left(\frac{\alpha-\beta}{n\alpha^2}\right)(H-2\beta |\nabla l|^2-2\gamma e^{2l}-2\varphi)-2be^{2l}\bigg]\\
&+\varphi_t-\Delta \varphi - 2\nabla l \cdot \nabla \varphi+2|\nabla l|^2e^{2l}\bigg[2\left(\frac{\alpha-\beta}{n\alpha^2}\right)\beta \gamma -2\alpha b-\beta b-3\gamma\bigg]\nonumber\\
&+ |\nabla l|^2\varphi\bigg[4\left(\frac{\alpha-\beta}{n\alpha^2}\right)\beta\bigg]+e^{2l}\bigg[4\left(\frac{\alpha-\beta}{n\alpha^2}\right)\gamma \varphi+2b\varphi+2a\gamma\bigg]\nonumber\\
&+2\left(\frac{\alpha-\beta}{n\alpha^2}\right)\bigg[\beta^2|\nabla l|^4 + \gamma^2e^{4l}+\varphi^2\bigg]\nonumber.
\end{align}
\end{lemma}

\begin{proof} We will achieve our result by applying the Cauchy-Schwarz inequality in the form of $|\nabla \nabla l|^2 \geq \frac{1}{n}(\Delta l)^2$ and also by substituting $\Delta l = \frac{1}{\alpha}(H - \beta |\nabla l|^2 - \gamma e^{2l} - \varphi)$. Upon doing this, we receive the following:
\begin{eqnarray*}\label{hterm}
\sq H&\geq& 2\nabla l \cdot \nabla H + 2\frac{(\alpha-\beta)}{n \alpha^2}(H - \beta |\nabla l|^2 - \gamma e^{2l} - \varphi)^2 + \varphi_t - \Delta \varphi-2\nabla l \cdot \nabla \varphi\\
&-& 2be^{2l}\bigg[(H - \varphi) + 2\alpha |\nabla l|^2 + \beta |\nabla l|^2-\gamma\frac{a}{b} + 3\frac{\gamma}{b}|\nabla l|^2\bigg]\\
&=& 2\nabla l \cdot \nabla H + H\bigg[2\left(\frac{\alpha-\beta}{n\alpha^2}\right)(H-2\beta |\nabla l|^2-2\gamma e^{2l}-2\varphi)-2be^{2l}\bigg]\\
&+& 2\left(\frac{\alpha-\beta}{n\alpha^2}\right)\bigg[\beta^2|\nabla l|^4 + \gamma^2e^{4l}+\varphi^2\bigg]+2|\nabla l|^2e^{2l}\bigg[2\left(\frac{\alpha-\beta}{n\alpha^2}\right)\beta \gamma -2\alpha b-\beta b-3\gamma\bigg]\\
&+& |\nabla l|^2\varphi\bigg[4\frac{\alpha-\beta}{n\alpha^2}\beta\bigg]+e^{2l}\bigg[4\left(\frac{\alpha-\beta}{n\alpha^2}\right)\gamma \varphi+2b\varphi+2a\gamma\bigg]+\varphi_t-\Delta \varphi - 2\nabla l \cdot \nabla \varphi.
\end{eqnarray*}

This yields the desired inequality.
\end{proof}

\section{Proof of the Main Theorem}

We now proceed to prove our main theorem. We apply the parabolic maximum principle by assuming for the sake of contradiction that there exists a first point $(z, t_0)$, $t_0 \neq 0$ at which $H(z,t_0) = 0$, which we show must occur in some compact region away from the origin.  At such a first time, the time derivative $H_t \leq 0$, the Laplacian $\Delta H \geq 0$, and the gradient $\nabla H = 0$ (vector). Our method of proof will be working with the time evolution of the right hand quantities to construct a contradiction of the form $0 \geq H_t(z, t_0) \geq A(z,t_0) > 0$ for some quantity $A$.  As a consequence of this contradiction, the quantity $H$ must be nonnegative for all space and time.\\
\indent Assume we are at the first point $(z, t_0)$ where $H=0$, at which $\nabla H$ is the 0 vector.  Therefore, by simplifying (9), we have:
\begin{eqnarray*}
\sq H &\geq& 2\left(\frac{\alpha-\beta}{n\alpha^2}\right)\bigg[\beta^2|\nabla l|^4\bigg]+2|\nabla l|^2e^{2l}\bigg[2\left(\frac{\alpha-\beta}{n\alpha^2}\right)\beta \gamma -2\alpha b-\beta b-3\gamma\bigg]\\
&+& 2\left(\frac{\alpha-\beta}{n\alpha^2}\right)\bigg[\gamma^2e^{4l}+\varphi^2\bigg]+|\nabla l|^2\varphi\bigg[4\left(\frac{\alpha-\beta}{n\alpha^2}\right)\beta\bigg] +e^{2l}\bigg[4\left(\frac{\alpha-\beta}{n\alpha^2}\right)\gamma \varphi+2b\varphi+2a\gamma\bigg]\\
&+&\varphi_t-\Delta \varphi - 2\nabla l \cdot \nabla \varphi
\end{eqnarray*}
From conditions $(a)$ and $(b)$, the first two terms are both nonnegative.  Thus:
\begin{eqnarray}
\sq H &\geq& 2\left(\frac{\alpha-\beta}{n\alpha^2}\right)\bigg[\gamma^2e^{4l}+\varphi^2\bigg]+\varphi_t - \Delta \varphi - 2\nabla l \cdot \nabla \varphi.\\
&+& |\nabla l|^2\varphi\bigg[4\left(\frac{\alpha-\beta}{n\alpha^2}\right)\beta\bigg]+e^{2l}\bigg[4\left(\frac{\alpha-\beta}{n\alpha^2}\right)\gamma \varphi+2b\varphi+2a\gamma\bigg]. \nonumber
\end{eqnarray}
Via application of the Cauchy-Schwarz inequality $a^2-2ab \geq -b^2$, we get:
$$|\nabla l|^2\varphi\bigg[4\left(\frac{\alpha-\beta}{n\alpha^2}\right)\beta\bigg]- 2\nabla l \cdot \nabla \varphi \geq -\frac{n\alpha^2|\nabla \varphi|^2}{4\beta(\alpha-\beta)\varphi}.$$

Hence,
\begin{equation}\label{split}
\sq H \geq 2\left(\frac{\alpha-\beta}{n\alpha^2}\right)\bigg[\gamma^2e^{4l}+\varphi^2\bigg]+\varphi_t - \Delta \varphi\\
- \frac{n\alpha^2|\nabla \varphi|^2}{4\beta(\alpha-\beta)\varphi}+e^{2l}\bigg[4\left(\frac{\alpha-\beta}{n\alpha^2}\right)\gamma \varphi+2b\varphi+2a\gamma\bigg].
\end{equation}

We will now use Cauchy-Schwarz again:
$$2\left(\frac{\alpha-\beta}{n\alpha^2}\right)\gamma^2e^{4l}+2e^{2l}\bigg[2\left(\frac{\alpha-\beta}{n\alpha^2}\right)\gamma \varphi+b\varphi+a\gamma\bigg]\geq -\frac{n\alpha^2}{2(\alpha-\beta)\gamma^2}\bigg[\left(2\left(\frac{\alpha-\beta}{n\alpha^2}\right)\gamma+b\right)\varphi+a\gamma\bigg]^2.$$

Therefore, we arrive at the following:
\begin{eqnarray*}
\sq H &\geq& 2\left(\frac{\alpha-\beta}{n\alpha^2}\right)\varphi^2+\varphi_t - \Delta \varphi\\
&-& \frac{n\alpha^2|\nabla \varphi|^2}{4\beta(\alpha-\beta)\varphi}-\frac{n\alpha^2}{2(\alpha-\beta)\gamma^2}\bigg[\left(2\left(\frac{\alpha-\beta}{n\alpha^2}\right)\gamma+b\right)\varphi+a\gamma\bigg]^2.
\end{eqnarray*}

To simply our notation for our differential equation, we define the following constants:
$$\omega = \frac{1}{\alpha}\sqrt{\frac{2(\alpha-\beta)}{n}},$$
$$\mu = a\alpha\sqrt{\frac{n}{2(\alpha-\beta)}},$$
$$\nu = \frac{1}{\alpha}\sqrt{\frac{2(\alpha-\beta)}{n}}+\frac{\alpha b}{\gamma}\sqrt{\frac{n}{2(\alpha-\beta)}}.$$

Then, the above inequality becomes
\begin{equation}
\sq H \geq (\omega \varphi)^2 - (\mu+\nu\varphi)^2 + \varphi_t - \left(\Delta \varphi + \frac{1}{2\beta\omega^2}\frac{|\nabla \varphi|^2}{\varphi}\right).
\end{equation}

In the same fashion as \cite{cck15}, we can define our test function to have a spatially-dependent portion that is the sum of rational functions of the form:
$$\sum_{k=1}^n \left(\frac{c}{(x_k-p_k)^2}+\frac{c}{(q_k-x_k)^2}\right).$$
\indent In doing so, we accomplish the dual task of causing this differential term to be 0 by choosing an appropriate $c$ as well as ensuring that the test function blows up towards positive infinity at the boundary of the $n$-rectangle $R = \prod[p_i, q_i]$.  Therefore, we can ensure that any point at which $H(x,t) = 0$ occurs in a spatially compact region.  Then, we can take each $p_k \to -\infty$ and $q_k \to \infty$ to retrieve only the time-dependent part in the limiting case.  Therefore, we can choose our function $\varphi$ to be time dependent only, and focus on solving the following differential inequality:
\begin{equation}
(\omega \varphi)^2 - (\mu+\nu\varphi)^2 + \varphi_t > 0.
\end{equation}

In order to solve this differential equation, we first assume that inequality $(c)$ holds.  Then, we show that $\varphi(t)$ is a valid solution to this differential equation which possesses the properties we desire.  For:
$$\varphi(t) = \left(\dfrac{a\alpha}{1-e^{2at}}\right)\left(\dfrac{\gamma}{\alpha b}e^{2at}-\dfrac{\alpha \gamma n}{4\gamma(\alpha-\beta)+\alpha^2bn}\right) = \frac{\mu}{1-e^{2\mu\omega t}}\left(\frac{1}{\nu - \omega}e^{2\mu\omega t} - \frac{1}{\nu + \omega}\right).$$

By Lemma 4 of \cite{CLPW14}, we know that in the constant form this is a valid solution to this differential equation.  The only other behavior we desire is that $\varphi(t) > 0$ for all time and that $\varphi(t)$ diverges towards positive infinity as $t \to 0$, so that we can ensure that $H(x,t)$ starts off positive and therefore its first zero must be a negative time derivative.
\indent For any $t > 0$, we have:
\begin{align*}
\text{sign}(\varphi(t)) &= \text{sign}\left[\left(\dfrac{a\alpha}{1-e^{2at}}\right)\right]\text{sign}\left[\left(\dfrac{\gamma}{\alpha b}e^{2at}-\dfrac{\alpha \gamma n}{4\gamma(\alpha-\beta)+\alpha^2bn}\right)\right].\\
\end{align*}
The sign of the first term is certainly negative, as both $\alpha > 0$ and $a > 0$.  Furthermore, by application of inequality $(c)$, we see that the second term is negative at time $t = 0$, and since $\gamma < 0$, for any $t > 0$ this term is also negative.  Thus, the overall sign is positive, and $\varphi(t) > 0 \forall t$.\\
\indent We can observe further that the limit behavior of the function as $t \to 0$ can be broken down into two terms as well.  Thus, we observe:
$$\lim_{t \to 0} \left(\dfrac{a\alpha}{1-e^{2at}}\right) = -\infty.$$
Similarly, by applying inequality $(c)$ we can observe:
$$\lim_{t \to 0} \left(\dfrac{\gamma}{\alpha b}e^{2at}-\dfrac{\alpha \gamma n}{4\gamma(\alpha-\beta)+\alpha^2bn}\right)  = \left(\dfrac{\gamma}{\alpha b}-\dfrac{\alpha \gamma n}{4\gamma(\alpha-\beta)+\alpha^2bn}\right)< 0.$$

Therefore, the limit of the entire function $\varphi(t)$ as $t \to 0$ must be positive infinity, and the function exhibits the behavior we desire. Thus, we have the contradiction
$$0 \geq \sq H  \geq (\omega \varphi)^2 - (\mu+\nu\varphi)^2 + \varphi_t > 0.$$
This proves our theorem in the case that inequalities $(a)$, $(b)$, and $(c)$ hold.

Now, we assume that inequality $(c)$ does not hold.  In this case, we refer back to $\eqref{split}$:
\begin{eqnarray*}
\sq H &\geq& 2\left(\frac{\alpha-\beta}{n\alpha^2}\right)\bigg[\psi^2\bigg]+\psi_t - \Delta \psi- \frac{n\alpha^2|\nabla \psi|^2}{4\beta(\alpha-\beta)\psi}\\
&+&2\left(\frac{\alpha-\beta}{n\alpha^2}\right)\gamma^2e^{4l}+2e^{2l}\bigg[\left(2\left(\frac{\alpha-\beta}{n\alpha^2}\right)\gamma +b\right)\psi+a\gamma\bigg].
\end{eqnarray*}

If $(c)$ does not hold, that means $2\left(\dfrac{\alpha-\beta}{n\alpha^2}\right)\gamma+b > 0$, and thus for a sufficiently well-chosen $\psi(t)$ and for small $t$, we can ensure that both of the last two terms are positive and therefore ignore them both in our calculations.  So, we choose:
$$\psi_1(t) = \dfrac{n\alpha^2}{2(\alpha-\beta)t}, \ \ t \leq \dfrac{n\alpha^2}{2(\alpha-\beta)(-a\gamma)}\left(2\left(\dfrac{\alpha-\beta}{n\alpha^2}\right)\gamma+b\right) = T.$$
If this is the case, we claim that for any $t \leq T$, the last two terms are both non-negative.  Since $\psi(t)$ is a decreasing function, it suffices to check at $t = T$:
\begin{eqnarray*}
2e^{2l}\bigg[\left(2\left(\frac{\alpha-\beta}{n\alpha^2}\right)\gamma +b\right)\psi(T)+a\gamma\bigg] &=& 2e^{2l}\bigg[\left(\frac{n\alpha^2}{2(\alpha-\beta)}\right)\left(\frac{2(\alpha-\beta)(-a\gamma)}{n\alpha^2}\right) + a\gamma\bigg]\\
&=& 2e^{2l}\bigg[-a\gamma+a\gamma\bigg] = 0.
\end{eqnarray*}
Thus, we can ignore these last two terms, as well as ignoring the spatial terms once again, and solve the ordinary differential equation:
\begin{equation}
2\left(\frac{\alpha-\beta}{n\alpha^2}\right)\psi^2+\psi_t > 0.
\end{equation}
whose solution is given by $\psi(t)$ as desired.  This function is also positive for all time and approaches positive infinity as $t \to 0$.\\
\indent In the case that $t > T$, we cannot ignore the last two terms, and we must carry out the Cauchy-Schwarz approximation as was done in the last case, from which receive:
\begin{equation}
(\omega \psi)^2 - (\mu+\nu\psi)^2 + \psi_t > 0.
\end{equation}
with the same constants $\nu, \mu, \omega$ as defined earlier.  However, we must solve this time to be continuous and differentiable with $\psi(t)$ at $t = T$.  Thus, we get:
$$\psi_2(t) = \frac{-\mu(1+e^{2\mu\omega (t-T)})}{(\nu - \omega)e^{2\mu\omega (t-T)}+(\nu + \omega)} = an\alpha^2\left(\dfrac{-\gamma\left(e^{2a(t-T)}+1\right)}{n\alpha^2b\left(e^{2a(t-T)}+1\right)+4\gamma(\alpha-\beta)}\right).$$
This function is positive for all time $t > T$, as the numerator is positive since $\gamma < 0$, and the denominator is positive because inequality $(c)$ does not hold.  Furthermore:
$$\psi_1(T) = -a\gamma\left(\frac{1}{2\left(\frac{\alpha-\beta}{n\alpha^2}\right)\gamma+b}\right).$$
$$\psi_2(T) = \frac{-2an\alpha^2\gamma}{2n\alpha^2b+4\gamma(\alpha-\beta)} = \dfrac{-a\gamma}{2\left(\frac{(\alpha-\beta)}{n\alpha^2}\right)\gamma+b}.$$
$$\psi_1'(T) = \frac{-n\alpha^2}{2(\alpha-\beta)T^2} = \frac{-2n(\alpha-\beta)a^2\alpha^2\gamma^2}{(2(\alpha - \beta)\gamma + bn\alpha^2)^2}.$$
$$\psi_2'(T) = an\alpha^2\left(\dfrac{-\gamma\left(e^{2a(T-T)}+1\right)}{n\alpha^2b\left(e^{2a(T-T)}+1\right)+4\gamma(\alpha-\beta)}\right) = \frac{-2n(\alpha-\beta)a^2\alpha^2\gamma^2}{(n\alpha^2b +2\gamma(\alpha-\beta))^2}.$$
Thus, $\psi(t)$ is continuous, differentiable, and positive everywhere, and we have the contradiction
$$0 \geq \sq H  \geq (\omega \psi)^2 - (\mu+\nu\psi)^2 + \psi_t > 0.$$
This proves our theorem in the case that inequalities $(a)$, $(b)$, and $(d)$ hold.

\begin{remark}
$\psi(t)$ turns out to be exactly twice differentiable.
\end{remark}

\begin{remark}
It is worth noting that $\displaystyle\lim_{t \to \infty}{\varphi(t)} =\displaystyle\lim_{t \to \infty}{\psi(t)} = \displaystyle\frac{a}{b}|\gamma| = -\displaystyle\frac{a}{b}\gamma$. When estimating quantities using the Harnack, it is often useful to consider just the limiting case $t \to \infty$, allowing us to replace all occurrences of $\phi(t)$ and $\psi(t)$ with $-\displaystyle\frac{a}{b}\gamma$.
\end{remark}

\begin{remark}
There are situations in which we can obtain a simpler Harnack by choosing specific values of $\alpha$, $\beta$, or $\gamma$. If we choose $\gamma = -2nb$, we get that
$$H =\alpha \Delta l + \beta |\nabla l|^2 - 2nb e^{2l} + \dfrac{n\alpha^2}{2(\alpha-\beta)t} \geq 0.$$
\end{remark}
\section{Applications}

In this section we give several applications of our differential Harnack estimate. First, we integrate our Harnack along a space-time curve to derive a classical Harnack inequality. Then, we characterize traveling wave solutions and standing solutions to the Newell-Whitehead equation.

\subsection{Classical Harnack}
Here we use our differential Harnack estimate to prove a classical Harnack inequality, comparing values of a positive solutions at different points.

\begin{corollary}
Let $f$ be a positive solution to \eqref{eq:1}. Pick two points $(x_1, t_1), (x_2, t_2) \in \bR ^n \times [0,\infty)$ with $0<t_1 < t_2$. Then we have 
\begin{equation}
\frac{f(x_2,t_2)}{f(x_1,t_1)} \geq \exp \left\{\frac{-(x_2 - x_1)^2}{4(t_2 - t_1)}\right\} \cdot  \exp \left\{ a\left(1+\frac{n}{3}\right)(t_2-t_1)\right\} \cdot \left(\frac{1-e^{2at_1}}{1-e^{2at_2}}\right)^{2n/3}.
\end{equation}
\end{corollary}
\begin{proof}
Let $\Gamma$ be any space-time curve connecting $(x_1,t_1)$ and $(x_2,t_2)$, and define $l = \log f$ as before. Then we have
$$l(x_2, t_2) - l(x_1,t_1) = \int_{\Gamma} \left[ l_t + \nabla l \cdot \frac{dx}{dt} \right] \; dt.$$
Using the fact that $l_t = \Delta l + |\nabla l|^2 + a - be^{2l}$, we get
$$l(x_2, t_2) - l(x_1,t_1) = \int_{\Gamma} \left[ \Delta l + |\nabla l|^2 + a - be^{2l}+ \nabla l \cdot \frac{dx}{dt} \right] \; dt .$$
By the choice of our $\alpha$, $\beta$, and $\gamma$ (see below), it follows from our differential Harnack estimate that $\Delta l \geq \displaystyle\frac{-\beta}{\alpha} |\nabla l|^2 + \displaystyle\frac{-\gamma}{\alpha} e^{2l} + \displaystyle\frac{-1}{\alpha} \varphi(t)$. Thus, we get
$$l(x_2,t_2) - l(x_1,t_1) \geq \int_{\Gamma} \left[ |\nabla l|^2\left(1-\frac{\beta}{\alpha}\right) - e^{2l}\left(b + \frac{\gamma}{\alpha}\right) + a - \frac{\varphi}{\alpha} + \nabla l \cdot \frac{dx}{dt} \right] \; dt. $$
Applying the Cauchy-Schwarz Inequality $a^2 + 2ab \geq -b^2$ to the $\nabla l$ terms, we see that 
$$ |\nabla l|^2\left(1-\frac{\beta}{\alpha}\right) + \nabla l \cdot \frac{dx}{dt} \geq \frac{-1}{4}\left(\frac{\alpha}{\alpha-\beta}\right)\left(\frac{dx}{dt}\right)^2.$$
Thus
$$l(x_2, t_2) - l(x_1, t_1) \geq \int_{\Gamma} \left[-\frac{1}{4}\left(\frac{\alpha}{\alpha-\beta}\right)\left(\frac{dx}{dt}\right)^2- e^{2l}\left(b + \frac{\gamma}{\alpha}\right) + a - \frac{\varphi}{\alpha}\right] \; dt.$$

At this point, we may choose $\beta =0$ and $\gamma = -nb\alpha$, which implies $b + \displaystyle\frac{\gamma}{\alpha} \leq0$ and $4\gamma(\alpha - \beta) + n\alpha^2\beta = -3n\alpha^2b < 0$, thus we can simplify the above inequality to
\begin{equation}
l(x_2, t_2) - l(x_1, t_1) \geq \int_{\Gamma} \left[ -\frac{1}{4} \left(\frac{dx}{dt}\right)^2 + a - \frac{\varphi}{\alpha}\right] \; dt.
\end{equation}
Because $\Gamma$ is any space-time curve connecting $(x_1,t_1)$ and $(x_2,t_2)$, we can take the infimum over all such space-time paths to get
$$\int_{\Gamma}\left(\frac{dx}{dt}\right)^2 \; dt = \frac{(x_2 - x_1)^2}{t_2 - t_1},$$
and
$$\int_{\Gamma} \frac{\varphi}{\alpha} \; dt = \int_{t_1}^{t_2} \left(\frac{an}{e^{2at}-1}\right)\left(e^{2at}+\frac{1}{3}\right) \; dt = \frac{1}{3}an(t_1-t_2) + \frac{2}{3}n\log\frac{1-e^{2at_2}}{1-e^{2at_1}}.$$
Thus we get:

$$l(x_2,t_2) - l(x_1,t_1) \geq \frac{-1}{4}\frac{(x_2 - x_1)^2}{t_2 - t_1} + \left(a+\frac{1}{3}an\right)(t_2-t_1) + \frac{2}{3}n\log\frac{1-e^{2at_1}}{1-e^{2at_2}}.$$

Exponentiate both sides to arrive at Corollary $4.1$.
\end{proof}
\subsection{Traveling Wave Solutions}
We call $f$ a traveling wave solution of $(1)$ if it is of the form
$$f(x,t) = f(x_1, x_2, \dots, x_n, t) = v(x_1, x_2, \dots, x_n + \eta t),$$
for some function $v: \bR^n \to \bR$ (see \cite{mihai2015}). Traveling wave solutions to the Newell-Whitehead equation are used to model traveling wave convection in binary fluids, and other forms of oscillatory instability (see \cite{malomed}). We use our differential Harnack to derive a lower bound for $\eta$, the wavespeed, of a positive traveling wave solution.

\begin{corollary}
Let $f(x,t) = v(x_1, x_2, \dots, x_n + \eta t)$ be a positive traveling wave solution. Suppose that $v(z) \to 0$ for some $z$ such that $|z| \to \infty$. Then we have
$$ \eta^2 \geq \frac{4}{3}a.$$
\end{corollary}
\begin{proof}
We start by rewriting our Harnack quantity so that it is in terms of $f$, instead of $l$. From our original estimate, we have
$$\alpha \Delta l + \beta |\nabla l|^2 +\gamma e^{2l} + \varphi(t) \geq 0.$$
Recalling that $l = \log f$ and $\Delta f = f_t - af + bf^3$, we get that
\begin{equation}
\alpha \frac{f_t}{f} - \alpha a + (\beta - 
\alpha)\frac{|\nabla f|^2}{f^2} + (\gamma + \alpha b) f^2 + \varphi(t) \geq 0.
\end{equation}
This is our revised Harnack estimate. In the case that $f(x,t) = v(x_1, x_2, \dots, x_n + \eta t)$ is a traveling wave solution, we get that 
$$\alpha\eta  \frac{v_{x_n}}{v} - \alpha a + \alpha b v^2 + (\beta - \alpha)\frac{|\nabla v|^2}{v^2} + \gamma v^2 + \varphi(t) \geq 0.$$
Notice that $|v_{x_n}| \leq |\nabla v|$. Applying  Cauchy-Schwarz again then yields that 
$$(\alpha - \beta)\displaystyle\frac{|\nabla v|^2}{v^2} - \alpha \eta \displaystyle\frac{|\nabla v|}{v} \geq - \displaystyle\frac{(\alpha \eta)^2}{4(\alpha - \beta)}. $$
Thus our inequality becomes
$$\frac{(\alpha \eta)^2}{4(\alpha - \beta)} \geq (a\alpha - \varphi) - (b\alpha + \gamma)v^2.$$
Because this inequality holds for any $t$, we can simplify by considering just the limiting cases, where $v \to 0$, and, by Remark 3.2, $\displaystyle\lim_{t \to \infty}{\varphi(t)} = \displaystyle\lim_{t \to \infty}{\psi(t)} = -\displaystyle\frac{a}{b} \gamma$. Rearranging then gives us our bound on $\eta$:
$$\eta^2 \geq a\left(\alpha + \frac{\gamma}{b}\right)\frac{4(\alpha-\beta)}{\alpha^2}.$$
To maximize the right hand side, we choose $\beta=0$ and $\gamma = -\displaystyle\frac{2}{3}b\alpha$, giving us Corollary $4.2$.
\end{proof}
We now use our differential Harnack estimate to prove a gradient estimate for the traveling wave solutions to the Newell-Whitehead equation.
\begin{corollary}
Let $f(x,t) = v(x_1, x_2, \dots, x_n + \eta t)$ be a positive traveling wave solution. Then we have
$$|\nabla v| \leq v\eta.$$
\end{corollary}

\begin{proof}
We start with (18), our Harnack in terms of $f$,:
$$\alpha \frac{f_t}{f} - \alpha a + (\beta - 
\alpha)\frac{|\nabla f|^2}{f^2} + (\gamma + \alpha b) f^2 + \varphi(t) \geq 0.$$
Now, from $f(x,t) = v(x_1, x_2, \dots, x_n + \eta t)$, we use that $f_t = \eta v_{x_n} \leq \eta |\nabla v|$ and again take the limiting case $\varphi, \psi \to -\displaystyle\frac{a}{b}\gamma$ to get 
$$(\gamma+ \alpha b)\left(v^2 - \frac{a}{b}\right) \geq (\alpha - \beta)\frac{|\nabla v|^2}{v^2} - \alpha \eta \frac{|\nabla v|}{v}.$$
By choosing $\beta = 0$ and $\gamma = -b\alpha$, we reduce the expression to
$$0 \geq \frac{|\nabla v|^2}{v^2} - \eta \frac{|\nabla v|}{v}.$$
Simplification yields Corollary $4.3$.
\end{proof}

\subsection{Standing Solutions}
We call a solution $f$ a \textit{standing solution} if $f_t=0$. 
\begin{corollary}
All positive standing solutions are constant.
\end{corollary}
\begin{proof}
We begin with (18):
$$\alpha \frac{f_t}{f} - \alpha a + (\beta - 
\alpha)\frac{|\nabla f|^2}{f^2} + (\gamma + \alpha b) f^2 + \varphi(t) \geq 0.$$
At this point, we assume $f_t=0$. We also again take the limiting case where  $\varphi, \psi \to -\displaystyle\frac{a}{b}\gamma$. Thus we have
$$- \alpha a + (\beta - \alpha)\frac{|\nabla f|^2}{f^2} + (\gamma+b\alpha) f^2 - \frac{a}{b}\gamma \geq 0.$$
At this point we rearrange and factor to get
$$|\nabla f|^2 \leq \frac{f^2}{\alpha - \beta}(b\alpha + \gamma)\left(f^2 - \frac{a}{b}\right).$$
Choosing $\gamma = -b\alpha$, the right hand side becomes $0$, giving us $| \nabla f|=0$. Because we have $| \nabla f | = f_t = 0$, we conclude that $f$ is constant.
\end{proof}
\def\cprime{$'$}

\paragraph{\textbf{Acknowledgements:}} D. Booth and J. Burkart's research were
supported by NSF through the Research Experience for Undergraduates Program at Cornell University, grant-1156350. Z.  Munro and J. Snyder's research were
supported by Cornell University Summer Program for Undergraduate Research. X. Cao's research was partially supported by a grant from the Simons Foundation (\#280161). The authors would like to thank Professor Robert Strichartz for his encouragement.

\bibliographystyle{plain}
\bibliography{bio}

\end{document}